\documentclass[11pt, a4paper, reqno]{amsart}

\usepackage[mathscr]{eucal}

\usepackage{tikz}
\newcommand*{\DashedArrow}[1][]{\mathbin{\tikz [baseline=-0.25ex,-latex, dashed,#1] \draw [#1] (0pt,0.5ex) -- (1.3em,0.5ex);}}%

\usepackage{mathrsfs}
\usepackage[all]{xy}
\usepackage{graphicx}
\usepackage{amssymb}
\usepackage{color}

\usepackage{amsmath}
\swapnumbers

\usepackage{hyperref}
\hypersetup{
    colorlinks=true, 
    linkcolor=blue, 
    urlcolor=red, 
}

\usepackage{enumerate}
\usepackage{enumitem}

\usepackage{geometry}                
\usepackage{verbatim}

\newfont{\sheaf}{eusm10 scaled\magstep1}

\def\N{\ensuremath{\mathbb N}}
\def\Z{\ensuremath{\mathbb Z}}
\def\Q{\ensuremath{\mathbb Q}}

\def\C{\ensuremath{\mathbb C}}

\def\P{\ensuremath{\mathbb P}}
\def\S{\ensuremath{\mathbb S}}

\def\cO{\ensuremath{\mathcal O}}
\def\cF{\ensuremath{\mathcal F}}

\def\cJ{\ensuremath{\mathcal J}}
\def\cI{\ensuremath{\mathcal I}}
\def\\L{\ensuremath{\mathcal L}}

\def\L{\ensuremath{\mathbf L}}

\def\d{\mathbf {d}}

\def\F{\mathbf {F}}
\def \lam{\mathbf {\lambda}}

\def\im{\operatorname{im}}
\def\rank{\operatorname{rank}}

\def\rank{\operatorname{rank}}

\newtheorem{THM}{Theorem}[section]

\newtheorem{theorem}{Theorem}
\newtheorem{proposition}[theorem]{Proposition}

\newtheorem{lemma}[theorem]{Lemma}
\newtheorem{definition}[theorem]{Definition}
\newtheorem{remark}[theorem]{Remark}
\newtheorem{corollary}[theorem]{Corollary}

\numberwithin{equation}{section}

\begin{document}


\title  [Stability of  logarithmic differential one-forms.]    {Stability of  logarithmic differential one-forms.}

\author[Fernando Cukierman, Javier Gargiulo Acea and C\'esar Massri.]{Fernando Cukierman, \\\\ Javier Gargiulo Acea, \\\\ C\'esar Massri.}


\begin{abstract}
This article deals with the irreducible components of the space of codimension one foliations in a projective space  defined by logarithmic forms of a certain degree. We study the geometry of the natural parametrization of the logarithmic components and we give a new proof of the stability of logarithmic foliations, obtaining also that these irreducible components are reduced.
\end{abstract}



\maketitle

\tableofcontents

\newpage

\noindent
\section{Introduction.}   \label{introduction}
We consider differential one-forms of logarithmic type
$\omega = F \ \sum_{i=1}^m \lambda_i  \  dF_i /F_i$   where, for $i = 1, \dots, m$, $F_i$ is a  homogeneous polynomial of a fixed degree $d_i$ in variables $x_0, \dots, x_n$,  with complex coefficients, $F= \prod_{j}  F_j$, and $\lambda_i$ are complex numbers such that $\sum_{i} d_i \lambda_i = 0$. 
Such an $\omega$ defines a global section of $\Omega_{\P^n}^1(d)$ for $d = \sum_{i} d_i$. Also, $\omega$ satisfies the Frobenius integrability condition $\omega \wedge d\omega = 0$. 

Fixing $\d = (m; d_1, \dots, d_m)$ denote $L_n(\d) \subset H^0(\P^n, \Omega^1_{\P^n}(d))$ the collection of all such logarithmic one-forms and
$\\L_n(\d) \subset \P H^0(\P^n, \Omega^1_{\P^n}(d)) = \P^N$ the corresponding closed projective variety. It is easy to see that $\\L_n(\d)$ is an irreducible algebraic variety.
Also, $\\L_n(\d)$ is contained in the subvariety $\mathcal F_n(d) \subset  \P^N$ of integrable one-forms of degree $d$. Here the motivating problem is to describe the irreducible components of $\mathcal F_n(d)$.

It was proved by Omegar Calvo in \cite{Omegar} that, for any $\d$, the variety of logarithmic forms  $\\L_n(\d)$ is an irreducible component of the moduli space $\mathcal F_n(d)$  of codimension one algebraic foliations of degree $d$ in $\P^n(\C)$. In other words, the logarithmic one-forms enjoy a stability condition among integrable forms. Actually, the results of \cite{Omegar} hold for more general ambient varieties than projective spaces.  

In this article we will provide another proof of O. Calvo's theorem, in case the ambient space is a complex projective space. 
Our strategy will be to calculate the tangent space $T(\omega)$ of $\mathcal F_n(d)$ at a general point $\omega \in \\L_n(\d)$.
The main results are stated in Theorems \ref{main} and \ref{main2}.

This method is completely algebraic and provides further information, especially the fact that $\mathcal F_n(d)$ results \emph{generically reduced}
along the irreducible component $\\L_n(\d)$.

The logarithmic components are the closure of the image of a multilinear map $\rho$, defined in Section 4, from a product of projective spaces into a projective space. We describe the base locus of $\rho$ in Section 5, and study its generic injectivity in Section 6. Our proof requires a detailed analysis of the derivative of $\rho$, started in Section 7.  Another important ingredient is the resolution of the ideal of various strata of the singular scheme of a logarithmic form; this is carried out in Section 8. The end of the proof is achieved in Section 9, where we distinguish two cases, depending on whether or not $\d$ is balanced.

\medskip

\medskip

\thanks{We thank Jorge Vit\'orio Pereira, Ariel Molinuevo and Federico Quallbrunn for several conversations at various stages of this work.}

\newpage

\noindent

\section{Notation.}    \label{notation}

We shall use the following notations:

\begin{flushleft}

\medskip

$\C^{n+1}$ = complex affine space of dimension $n+1$.

\medskip

$\P^n$ = complex projective space of dimension $n$.

\medskip

$S_n = \C[x_0, \dots, x_n]$ = graded ring of polynomials with complex coefficients in $n+1$ variables.

When $n$ is understood we denote $S_n = S$.

\medskip

$S_n(d)$ = homogeneous elements of degree $d$ in $S_n$.

When $n$ is understood we denote $S_n(d) = S(d)$.

Recall that one has $S_n(d) = H^0(\P^n, \mathcal O_{\P^n}(d))$.

\medskip

$\Omega^q_X$ = sheaf of algebraic differential $q$-forms on an algebraic variety $X$.

$\Omega^q(X)$ = the set of rational $q$-forms on $X$ (with $X$ an irreducible variety).

It is a vector space over the field $\C(X)$ of rational functions of $X$.

\medskip

$\Omega^q_n = H^0(\C^{n+1}, \Omega^q_{\C^{n+1}})$.

A typical element of $\Omega^1_n$ is $\omega = \sum_{i=0}^n a_i \ dx_i$ with $a_i \in S_n$.

More generally, a typical element of $\Omega^q_n$ may be written in the usual way as 
\linebreak
$\sum_{|J|=q} a_J \ dx_J$ with $a_J \in S_n$ and $dx_J = dx_{j_1} \wedge \dots \wedge dx_{j_q}$
where $J =\{j_1, \dots, j_q\}$ with $j_1 < \dots < j_q$.

When $n$ is understood we denote $\Omega^q_n = \Omega^q$.

$\Omega^q_n$ is a graded $S_n$-module with homogeneous piece of degree $d$ defined by

$\Omega^q_n(d) = \{\sum_{|J|=q} a_J \ dx_J, \ a_J \in S_n(d-q)$\}.

In particular, $dx_i$ is homogeneous of degree one.

The exterior derivative is an operator of degree zero, i. e. it preserves degree. 

\medskip

$H^0(\P^n, \Omega^1_{\P^n}(d))$ = projective one-forms of degree $d$.

It follows from the Euler exact sequence that $ \omega = \sum_{i} a_i dx_i \in \Omega^1_n(d)$ is projective if and only if it contracts to zero with the Euler or radial vector field $R =  \sum_{i=0}^n x_i \frac{\partial}{\partial x_i}$, that is, if $\sum_{i} a_i x_i = 0$.

\medskip

$\P^n(d) = \P(H^0(\P^n, \Omega^1_{\P^n}(d)))$.

\medskip

$F_n(d) = \{ \omega \in H^0(\P^n, \Omega^1_{\P^n}(d))  /  \omega \wedge d \omega = 0\}$ = 
the set of integrable projective one-forms in $\P^n$  of degree $d$, and 

$\mathcal F_n(d) \subset \P^n(d)$ the projectivization of $F_n(d)$.

\medskip

$\P^n(\d) = \P \Lambda(\d) \times \prod_{i=1}^m \P S_n(d_i)$.

\end{flushleft}

\medskip

\medskip

\newpage

\section{Logarithmic one-forms.}    \label{logarithmic}

\begin{definition} \label{definition0}
Fix natural numbers $n, d$ and  $m$. Let
$$\d = (m; d_1, \dots, d_m)$$
be a partition of $d$ into $m$ parts, that is, for $i = 1, \dots, m$ each $d_i$ is a natural number and $\sum_{i=1}^m d_i =d$. 
Let us normalize so that $d_i \ge d_{i+1}$ for all $i < m$.
We denote 
$$P(m, d)$$ 
the set of all such partitions of $d$ into $m$ parts.
\end{definition}

\medskip

\begin{definition} \label{definition1}
Fix $\d = (m; d_1, \dots, d_m) \in P(m, d)$. A differential one-form $\omega \in \Omega^1_n$ is \emph{logarithmic of type} $\d$ if 
$$\omega = (\prod_{j=1}^m  F_j)   \sum_{i=1}^m \lambda_i  \  dF_i /F_i = 
 \sum_{i=1}^m \lambda_i   \ (\prod_{j \ne i}  F_j)  \  dF_i $$
where $F_i \in S_n(d_i)$ is a non-zero homogeneous polynomial of degree $d_i$ and the $\lambda_i$ are complex numbers.
\end{definition}

\begin{definition} \label{definition notation}
It will be convenient to use the following notation. For 
$\d$ and $F_i \in S_n(d_i)$  as above, 
$$\F = (F_1, \dots, F_m), \ \ \ \ F = \prod_{j=1}^m F_j,$$
$$ \hat{F}_i =  \prod_{j \ne i}  F_j = F / F_i, \ \ \ \  
\hat{F}_{i j} =  \prod_{k \ne i, k \ne j}  F_k = F / F_i F_j, \ (i \ne j),$$ 
or, more generally, for a subset $A \subset \{1, \dots, m\}$ we write
$$\hat{F}_A =  \prod_{j \notin A} F_j$$
Hence a logarithmic one-form may be written
\begin{equation}
\omega = F \  \sum_{i=1}^m \lambda_i   \  dF_i /F_i = 
 \sum_{i=1}^m \lambda_i  \  \hat{F}_i  \  dF_i . \label{log}
 \end{equation}
 
 \noindent
 We denote $\hat d_i =  \sum_{j \ne i} d_j$ the degree of $\hat F_i$ and, more generally,
 $\hat d_A =   \sum_{j \notin A} d_j$   the degree of $ \hat{F}_A $.
\end{definition}

\begin{proposition} \label{proposition1} For $\omega$ a logarithmic one-form as above,

a) $\omega$ is homogeneous of degree $d=\sum_{i=1}^m d_i$. 

b) $\omega$ is integrable.

c) $ <R, \omega> = (\sum_{i=1}^m d_i \lambda_i) F $. In particular, $\omega$ is projective if and only if 
$$\sum_{i=1}^m d_i \lambda_i = 0.$$
\end{proposition}

\begin{proof} 

a) Since the exterior derivative is of degree zero, each term in the sum $ \sum_{i=1}^m \lambda_i  \  \hat{F}_i  \  dF_i $ is homogeneous of degree $d$, hence the claim.

b) For each polynomial $G$, the rational one-form $dG / G$ is closed. It follows that $\omega / F = \sum_{i=1}^m \lambda_i   \  dF_i /F_i $ is closed, hence integrable. A short calculation shows that the product of a rational function with an integrable rational one-form is an integrable rational one-form. Therefore, $\omega = F \ \omega/F$ is integrable.

c) Euler's formula implies that $<R, dG> = e G$ for $G \in S_n(e)$. By linearity of contraction we have
$<R, \omega> = <R,  \sum_{i} \lambda_i  \  \hat{F}_i  \  dF_i > = \sum_{i}  d_i \lambda_i   \hat{F}_i   F_i = (\sum_{i} d_i \lambda_i) F$.

\end{proof}

\medskip

\begin{proposition} \label{proposition1.5} Suppose $\omega$ is logarithmic as in \ref{log}. Then,

a) $d\omega = ({dF} / {F}) \wedge \omega =  \sum_{1 \le i, j \le m}   \lambda_j  \ \hat {F}_{ij} \  dF_i \wedge dF_j =  \sum_{1 \le i < j \le m}   (\lambda_j - \lambda_i) \ \hat {F}_{ij} \  dF_i \wedge dF_j $. 

b) $F$ is an integrating factor of $\omega$:
$d(\omega / F) = 0$,
or, equivalently, $F d\omega - dF \wedge \omega = 0$.

c) Each hypersurface $F_i = 0$ is an algebraic leaf of $\omega$, that is, $dF_i / F_i \wedge \omega$ is a regular 2-form (i. e. without poles). Hence $dF_i \wedge \omega = 0$ on the hypersurface $F_i = 0$.
\end{proposition}

\begin{proof}
These follow by straightforward calculations, left to the reader.
\end{proof}

\medskip

\medskip

\section{The logarithmic components and their parametrization.}    \label{logarithmic component}

As before, we fix  natural numbers $n, d$ and  $m$ and a partition $\d = (m; d_1, \dots, d_m)$ of $d$.

For a complex vector space $V$ we denote $\P V = V -\{0\} / \C^*$ the corresponding projective space of one-dimensional subspaces of $V$. 
Let $\pi: V -\{0\} \to \P V$ be the canonical projection. If $X \subset V$ we call $\P X = \pi(X -\{0\}) \subset \P V$ the projectivization of $X$.

As in Section \ref{notation}, we denote 
$$\P^n(d) = \P H^0(\P^n, \Omega^1_{\P^n}(d))$$
the projective space of sections of $\Omega^1_{\P^n}(d)$. This is the ambient projective space that contains the set of integrable forms $\cF_n(d)$ and the logarithmic components that we will investigate.

\begin{definition} \label{definition component}
Let $L_n(\d) \subset H^0(\P^n, \Omega^1_{\P^n}(d))$ denote the set of all logarithmic projective one-forms of type $\d$ in $\P^n$, 
and $\P L_n(\d) \subset \P^n(d)$ its projectivization. We denote 
$$\\L_n(\d) \subset \P^n(d)$$ 
the Zariski closure of $\P L_n(\d)$.
\end{definition}

If $\omega$ is a non-zero logarithmic form, the corresponding projective point $\pi(\omega)$ will  be denoted simply by $\omega$
when the danger of confusion is small.

Let
$$\Lambda(\d) =  \{(\lambda_1, \dots, \lambda_m) \in \C^m /  \sum_{i=1}^m d_i \lambda_i = 0\}$$
which is a hyperplane in $\C^m$.

\begin{definition} \label{parametrization}
Consider the map
$$\mu: V_n(\d) := \Lambda(\d) \times \prod_{i=1}^m S_n(d_i)  \to H^0(\P^n, \Omega^1_{\P^n}(d))$$ 
such that 
$$\mu((\lambda_1, \dots, \lambda_m), (F_1, \dots, F_m)) = \sum_{i=1}^m \lambda_i  \  \hat{F}_i  \  dF_i$$
 and 
$$\rho: \P^n(\d) := \P \Lambda(\d) \times \prod_{i=1}^m \P S_n(d_i) \   \DashedArrow[->,densely dashed    ]  \  \P^n(d) = \P H^0(\P^n, \Omega^1_{\P^n}(d)) $$
such that
$$\rho(\pi(\lambda_1, \dots, \lambda_m), (\pi(F_1), \dots, \pi(F_m))) = \pi(\sum_{i=1}^m \lambda_i  \  \hat{F}_i  \  dF_i).$$
\end{definition}

\begin{remark} \label{remark2} 
a) $\mu$ is a multi-linear map. By Proposition \ref{proposition1}, the image of $\mu$ is $L_n(\d)$. 
\newline
\noindent
b) The induced map $\rho$ from a product of projective spaces into a projective space is only a rational map. Later we will determine the base locus $B(\rho) =\{(\pi(\lambda), \pi(F)) / \mu(\lambda, F) = 0\}$ of $\rho$. Anyway, it is clear that the image of $\rho$ is $\P L_n(\d)$. Hence $\\L_n(\d)$ is the closure of the image of $\rho$. Therefore, $\\L_n(\d)$ is a projective irreducible variety.

\end{remark}

\medskip

\medskip

\section{Base locus.}    \label{ }

Let $B(\mu) = \mu^{-1} (0)$. Then $B(\mu)  \subset V_n(\d)$ is an affine algebraic set, and we intend to describe its irreducible components.

Let us remark that the multilinearity of $\mu$ implies that $B(\mu)$ is stable under the natural action of $({\C^*})^{m+1}$ on $V_n(\d)$. 

From the multilinearity of $\mu$ it follows that $Z = \{(\lam, \F) \in V_n(\d)/ \lam = 0 \mathrm{ \ or \ } F_i = 0 \mathrm{ \ for \ some \ } i \}$ is contained in $B(\mu)$. We denote $B = B(\mu) - Z$ and
$$B(\rho) = \pi(B) \subset \P^n(\d)$$
the base locus of $\rho$. 

An example of a point in the base locus is the following. Suppose $d_1 = \dots = d_m$. It is then clear that if $F_1 = \dots = F_m$ then $(\lam, \F) \in B(\mu)$. More generally, each string of equal $d_i$'s gives elements of $B(\mu)$: if $d_i = d_j$ for all $i, j \in A$, where $A \subset \{1, \dots, m\}$, then taking $F_i = F_j$ for all $i, j \in A$, $\sum_{i \in A} d_i \lambda_i =0$, $\lambda_j = 0$ for $j \notin A$, we obtain that $(\lam, \F) \in B(\mu)$.

These examples generalize as follows: suppose our $d_i$'s may be written as 
\begin{equation}
d_i = \sum_{j=1}^{m'} e_{ij}  d'_j, \ \ \ i = 1, \dots, m, \label{factorization}
\end{equation}
where $m' \in \N$,  $d'_j \ge 1$ and $e_{ij} \ge 0$ are integers. 
Let  $\lam \in \Lambda_n(\d)$ such that 
$\sum_{i=1}^{m} e_{ij} \lambda_i = 0$ for $j=1, \dots, m'$, and take $\F$ such that
\begin{equation}
F_i = \prod_{j=1}^{m'} G_j^{e_{ij}}   \label{factorization2}
\end{equation}
for some $G_j \in S_n(d'_j)$, $j=1, \dots , m'$. Then,
\begin{equation}
\sum_{i=1}^m \lambda_i  \  dF_i /F_i = \sum_{i=1}^m \lambda_i  \sum_{j=1}^{m'} e_{ij}  \  dG_j /G_j =
\sum_{j=1}^{m'}  (\sum_{i=1}^m  \lambda_i  e_{ij})  \  dG_j /G_j = 0 \label{base}
\end{equation}
and we obtain elements in the base locus.

\medskip

\noindent
We will see now that this construction accounts for all the irreducible components of the base locus.

\medskip

\begin{definition} \label{definition F}
We denote $F(\d)$ the collection of all decompositions of $\d$ as  in \ref{factorization}, that is,  let
$$F(\d) = \{(m', e, \d') / \  m' \in \N, \ e \in \N^{m \times m'}, \ \d' \in (\N-\{0\})^{m'}, \  \d=e \ \d', \ e \mathrm { \  without \ zero \ columns \ } \}$$
\end{definition}

\noindent
In \ref{factorization}, for each $i$ there exists $j$ such that $e_{ij} > 0$; that is, all rows of $e$ are non-zero. 
This follows from $d_i > 0$. 
If the $j$-th column of $e$ is zero then in the decomposition  \ref{factorization} the terms
$e_{ij} d'_j$ are zero and do not contribute, so this zero column may be disregarded.

\medskip

\noindent
Let us remark that $F(\d)$ is finite: we have,
$d = \sum_i d_i = \sum_{i, j} e_{ij} d'_j \ge \sum_j d'_j \ge m'$, hence $m'$ is bounded. Also, \ref{factorization} implies $e_{ij} \le d_i / d'_j \le d_i$, so
all $e_{ij}$ are also bounded.

\medskip

\noindent
For $\varphi = (m', e, \d') \in F(\d)$
denote the (Segre-Veronese) map
$$\nu_{\varphi}: \prod_{j=1}^{m'} S_n(d'_j)  \to \prod_{i=1}^m S_n(d_i)$$
$$\nu_{\varphi}(G_1, \dots, G_{m'}) = (F_1, \dots, F_m)$$
such that $F_i = \prod_{j=1}^{m'} G_j^{e_{ij}}$. Also, let 
$$\Lambda(e) = \{\lam \in \Lambda(\d) / \lam \ e = 0\}$$
which is a linear subspace of $\C^m$ of dimension  $m - \rank(e)$. 

\noindent
Notice  that $\lam \ e = 0$ implies  $\lam \ \d = 0$. For $\varphi \in F(\d)$ let
$$B_{\varphi} =  \Lambda(e) \times \im \nu_{\varphi} \subset V_n(\d)$$
By the calculation \ref{base} we know that $B_{\varphi} \subset B(\mu)$ for all $\varphi \in F(\d)$.

\noindent
Each $B_{\varphi}$ is clearly irreducible. Next we will see, first, that $B(\mu) = Z \cup \bigcup_{\varphi \in F(\d)} B_{\varphi}$. 
And, second, we will determine when there are inclusions among the $B_{\varphi}$'s, thus characterizing the irreducible components of the base locus.

\noindent
Let us first recall  from \cite{Joua}, Lemme 3.3.1, page 102, the following

\begin{proposition} \label{lemajoua}
Let $F_i \in S_n(d_i), \ i=1, \dots, m$, be irreducible distinct  (modulo multiplicative constants) homogeneous polynomials. 
If $\lambda_i \in \C$ are such that 
$$\sum_{i=1}^m \lambda_i \ dF_i/F_i = 0$$
 then $\lambda_i = 0$ for all $i$. That is, the rational one-forms $dF_1 / F_1, \dots, dF_m / F_m$ are linearly independent over $\C$.
\end{proposition}

\begin{corollary} \label{}
Let $(\lam, \F) \in V_n(\d)$ with the $F_i$ distinct and irreducible, and $\lam \ne 0$. Then $(\lam, \F) \notin B(\mu)$.
\end{corollary}

\begin{proposition} \label{propbase}
With the notations above, we have $B(\mu) = Z \cup \bigcup_{\varphi \in F(\d)} B_{\varphi}$.
\end{proposition}

\begin{proof}
Let $(\lam, \F) \in B = B(\mu) - Z$. Write each $F_i$ as a product of distinct irreducible homogeneous polynomials:
$$F_i = \prod_{j=1}^{m'} G_j^{e_{ij}}$$
We allow some $e_{ij} = 0$. Denote $d'_j$ the degree of $G_j$. Taking degree we obtain $\d = e \ \d'$. 
Repeating the calculation of \ref{base} we have
\begin{equation}
0 = \sum_{i=1}^m \lambda_i  \  dF_i /F_i = \sum_{i=1}^m \lambda_i  \sum_{j=1}^{m'} e_{ij}  \  dG_j /G_j =
\sum_{j=1}^{m'}  (\sum_{i=1}^m  \lambda_i  e_{ij})  \  dG_j /G_j  \label{base2}
\end{equation}
Since the $G_j$ are irreducible, Proposition \ref{lemajoua} implies that $\sum_{i=1}^m  \lambda_i  e_{ij} = 0$
for all $j = 1, \dots, m'$. Therefore, $(\lam, \F) \in B_{\varphi}$ with $\varphi = (m', e, \d') \in F(\d)$, as claimed.
\end{proof}

\noindent
Regarding possible inclusions among the $B_{\varphi}$'s, we make the following

\begin{definition} 
For $\varphi_1 = (m_1, e_1, \d_1), \  \varphi_2 = (m_2, e_2, \d_2) \in F(\d)$ we write $\varphi_2 \le \varphi_1$ if 
$\rank(e_1) = \rank(e_2)$ and there exists $e_3 \in \N^{m_1 \times m_2}$ such that $e_2 = e_1 \ e_3.$ 
\end{definition}

\noindent
Then we have

\begin{proposition} 
For $\varphi_1, \varphi_2 \in F(\d)$, $B_{\varphi_2} \subset B_{\varphi_1}$ if and only if  $\varphi_2 \le \varphi_1$.
\end{proposition}
\begin{proof}
Suppose $B_{\varphi_2} \subset B_{\varphi_1}$. Choose an element $(\lam, \F)  \in B_{\varphi_2}$, that is, $\lam \ e_2= 0$ and
$F_i = \prod_{k=1}^{m_2} {H}_k^{{e_2}_{ik}}$ for all $i$, for some $H_k$. We may take this element so that the $H_k$'s are irreducible. 
By our hypothesis, $(\lam, \F)  \in B_{\varphi_1}$ and we also have $F_i = \prod_{j=1}^{m_1} G_j^{{e_1}_{ij}}$ for all $i$, for some $G_j$.  By unique factorization and the irreducibility of the $H_k$,
$G_j = \prod_{k=1}^{m_2} {H}_k^{{e_3}_{jk}}$ for some ${e_3}_{jk} \in \N$. A simple calculation now gives $e_2 = e_1 \ e_3$.

Also, the  equality $e_2 = e_1 \ e_3$ just obtained easily implies $\Lambda(e_1)  \subset \Lambda(e_2)$. Since we are assuming
$B_{\varphi_2} \subset B_{\varphi_1}$, we also have  $\Lambda(e_2)  \subset \Lambda(e_1)$. 
Hence  $\Lambda(e_1) = \Lambda(e_2)$, and therefore $\rank(e_1) = \rank(e_2)$.

Conversely, suppose $\varphi_2 \le \varphi_1$.  Then $e_2 = e_1 \ e_3$ and $\rank(e_1) = \rank(e_2)$ imply, as before, that 
$\Lambda(e_1) = \Lambda(e_2)$. Also, the condition $e_2 = e_1 \ e_3$ easily implies that  $\im \nu_{\varphi_2}  \subset \im \nu_{\varphi_1}$.
Hence $B_{\varphi_2} \subset B_{\varphi_1}$.
\end{proof}

\medskip

\begin{corollary} 
The irreducible components of $B(\rho)$ are the $\pi(B_{\varphi})$ for $\varphi$ a maximal element of the finite ordered set  $(F(\d), \le)$.
\end{corollary}

\medskip

\medskip

\section{Generic injectivity.}    \label{ }

Suppose
$(\lam, \F),  (\lam', \F') \in V_n(\d)$ are such that $\mu(\lam, \F) = \mu(\lam', \F') \ne 0$, that is, 
$$F \  \sum_{i=1}^m \lambda_i   \  dF_i /F_i  = \omega = F' \  \sum_{i=1}^m \lambda'_i   \  dF'_i /F'_i.$$
Next we discuss conditions that imply that $(\lam, \F) = (\lam', \F')$.

\medskip

Let's observe that if the partition $\d$ contains repeated ${d_i}\ 's$ then the generic injectivity may hold only \emph{up to order}.
More precisely, suppose $A \subset \{1, \dots, m\}$ is such that $d_i = d_j$ for all $i, j \in A$. For each permutation 
$\sigma \in \S_m$ such that $\sigma(j) = j$ for $j \notin A$, clearly we have $\mu(\lam, \F) = \mu(\sigma. \lam, \sigma. \F)$
for all $(\lam, \F) \in V_n(\d)$. For $e \in \N$ let $A_e = \{i / d_i =e\}$. Then the non-empty $A_e$ form a partition of $\{1, \dots, m\}$.
Let $\S(e) = \{\sigma \in \S_m / \sigma(j) =j, \forall j \notin A_e\}$ and $\S(\d) = \prod_e \S(e)$.
Then the subgroup $\S(\d) \subset \S_m$ acts on $V_n(\d)$
and $\mu$ is constant on its orbits. By injectivity \emph{up to order} we will of course mean injectivity of the  induced map
with domain $V_n(\d)/\S(\d)$.

\begin{proposition} \label{injectivity} The rational map
$$\rho: \P^n(\d) \   \DashedArrow[->,densely dashed    ]  \  \\L_n(\d)  \subset  \P^n(d)$$
as in Definition \ref{parametrization}, is generically injective (up to order).
\end{proposition}

\begin{proof} We will prove the existence of a non-empty Zariski open $U \subset X$ such that $\rho|_U$ is injective morphism
(up to order). It is easy to see, using that $\rho$ is a dominant map of irreducible varieties, that  the existence of such a $U$ 
implies that there exists a non-empty Zariski open $V \subset  \\L_n(\d)$ such that $\rho: \rho^{-1}(V) \to V$ is injective (up to order).

Consider the Zariski open $\S(\d)$-stable $U \subset V_n(\d)$ of points $(\lam, \F)$ such that
the $F_i$ are irreducible and all distinct. Hence, for $(\lam, \F),  (\lam', \F') \in U$ distinct (up to order),
$F = \prod_i F_i \ne F' = \prod_i F'_i$. Suppose $\mu(\lam, \F) = \omega = \mu(\lam', \F') \ne 0$. 
Then $\omega$ has two integrating factors $F$ and $F'$, and therefore has a rational first integral $f = F / F'$.
It follows that $\omega$ has infinitely many algebraic leaves (the fibers of $f$). 

On the other hand, if $(\lambda_1 : \dots : \lambda_m) \in \P^{m-1}(\C) - \P^{m-1}(\Q)$, 
Proposition (3.7.8) from \cite{Joua} implies that $\omega$ has only finitely many algebraic leaves.

Let $U_0 = \{(\lam, \F) \in U / \lam \in  \P^{m-1}(\C) - \P^{m-1}(\Q)\}$. 

Consider the restriction $\rho: U \to \\L_n(\d)$ and $\tilde \rho: U/\S(\d) \to \\L_n(\d)$ the induced map.

We obtain that if  
$\omega = \mu(\lam, \F)$ with $(\lam, \F) \in U_0$ then $\tilde \rho^{-1} (\omega) = \{(\lam, \F)\}$. 

This implies, first, that since $\rho$ has a fiber of dimension zero, $\mathrm{dim} (U) = \mathrm{dim} (\\L_n(\d))$ and the general fiber of $\rho$ is finite.
Also, since the (open analytic) set $U_0$ is Zariski dense in $U$ (because $\C - \Q$ is dense in $\C$), $U_0$ is not contained in the branch
divisor of $\tilde \rho$ and hence $\tilde \rho$ has degree one, and therefore is birational, as claimed.

\end{proof}

\medskip

\medskip

\section{Derivative of the parametrization.}    \label{derivative}

With the notation of Definition \ref{parametrization}, let 
$$(\lam, \F) = ((\lambda_1, \dots, \lambda_m), (F_1, \dots, F_m)) \in V_n(\d)$$
be a point in the vector space  $V_n(\d)$ domain of $\mu$. 

Let  $(\lam', \F') = ((\lambda'_1, \dots, \lambda'_m), (F'_1, \dots, F'_m)) \in V_n(\d)$
represent a tangent vector 
$$(\lam, \F) + \epsilon (\lam', \F'), \ \  \epsilon^2 = 0,$$ 
to $V_n(\d)$ at $(\lam, \F)$.

From the multilinearity of $\mu$ we easily obtain the following formula for its derivative:
$$d\mu(\lam, \F): V_n(\d) \to H^0(\P^n, \Omega^1_{\P^n}(d))$$
\begin{equation}
d\mu(\lam, \F)(\lam', \F') =  \sum_{i} \lambda'_i  \  \hat{F}_i  \  dF_i +
\sum_{i \ne k} \lambda_i   \ F'_k \ \hat{F}_{i k}  \  dF_i + \sum_{i} \lambda_i  \  \hat{F}_i  \  dF'_i \label{derivative}
\end{equation}

\begin{remark} \label{remark3}
By Proposition \ref{proposition1} b), the image of $\mu$ is contained in the variety of integrable projective forms 
$F_n(d) \subset H^0(\P^n, \Omega^1_{\P^n}(d))$. Hence for each $(\lam, \F) \in V_n(\d)$ 
we have an inclusion of vector spaces
\begin{equation}
\im d\mu(\lam, \F) \subset T_{F_n(d)}(\omega) = \{\alpha \in  H^0(\P^n, \Omega^1_{\P^n}(d)) / 
\ \omega \wedge d\alpha +  \alpha  \wedge d\omega  = 0\} \label{tangents}
\end{equation}
where $\omega= \mu(\lam, \F)$ and $T_{F_n(d)}(\omega)$ denotes de tangent space of
$F_n(d)$ at the point $\omega$. 

Our main task in Section \ref{surjectivity} will be to show that this inclusion is actually an equality,
for a sufficiently general  $(\lam, \F) \in V_n(\d)$.
\end{remark}

\begin{definition} \label{eta}
It is convenient now to introduce the following notation: 

$\omega = \mu(\lam, \F) =  \sum_{i=1}^m \lambda_i  \  \hat{F}_i  \  dF_i$ \emph{(a logarithmic one-form)},

$\eta = \omega / F =  \sum_{i=1}^m \lambda_i  \  dF_i/ F_i$ \emph{(the corresponding rational logarithmic one-form)},

$\alpha = d\mu(\lam, \F)(\lam', \F') = \sum_{i} \lambda'_i  \  \hat{F}_i  \  dF_i +
\sum_{i \ne k} \lambda_i   \ F'_k \ \hat{F}_{i k}  \  dF_i + \sum_{i} \lambda_i  \  \hat{F}_i  \  dF'_i$,

$\beta = \alpha / F= \sum_{i} \lambda'_i    \  dF_i / F_i +
\sum_{i \ne k} \lambda_i   \ F'_k /F_k \  dF_i/F_i + \sum_{i} \lambda_i   \  dF'_i/F_i$.

\end{definition}

\medskip

\begin{proposition} \label{proposition2} With the notations above,  we have
$$\beta = \eta' + (G/F) \eta + d(H/F)$$
where 

$\eta'=  \sum_{i=1}^m \lambda'_i  \  dF_i/ F_i$,  
   
$G = \sum_{i=1}^m  \hat{F}_i  \  F'_i \in S_n(d)$, and

$H = \sum_{i=1}^m \lambda_i  \  \hat{F}_i  \  F'_i \in S_n(d)$.

\end{proposition}

\begin{proof} 
We add and substract to $\beta$ the sum $\sum_{i} \lambda_i  \  {F'}_i / F_i^2 \  dF_i$. 
A straightforward calculation gives the proposed expression.
\end{proof}

\medskip

\medskip

\section{Singular ideals of logarithmic one-forms and their resolution.}    \label{singularities}

For $\omega \in H^0(\P^n, \Omega^1_{\P^n}(d))$ denote $S(\omega) \subset \P^n$ the scheme of zeros of $\omega$
and $\cI= \cI_{\omega} \subset \cO_{\P^n}$ the corresponding ideal sheaf. Considering $\omega$ as a morphism 
$\cO_{\P^n} \to  \Omega^1_{\P^n}(d)$, $\cI$ is defined as the image of the dual morphism $T_{\P^n}(-d) \to \cO_{\P^n}$. 
Also, if $\omega = \sum_{i=0}^n a_i dx_i$ then $\cI$ corresponds to the homogeneous ideal generated by
$a_0, \dots, a_n \in S_n(d-1)$.

\medskip

We keep the notation of Definitions \ref{definition1} and \ref{definition notation}. 

Let $(\lam, \F) \in V_n(\d)$ and
$\omega = F . \sum_{i=1}^m \lambda_i   \  dF_i /F_i = \sum_{i=1}^m \lambda_i  \  \hat{F}_i  \  dF_i $ the corresponding
logarithmic one-form.

We denote
$$X_i = \{x \in \P^n / F_i(x)=0\}$$
the hipersurface defined by $F_i$. 

For $i \ne j$,
$$X_{ij} = X_i \cap X_j = \{x \in \P^n / F_i(x)=F_j(x) = 0\}$$
and, more generally, for a subset $A \subset \{1, \dots, m\}$,
$${X}_A =  \bigcap_{i \in A} X_i $$
For $1 \le r \le m$ we write
$$X^{(r)} = \bigcup_{|A| = r} X_A \,$$
and we shall use especially the following particular cases
$$X^{(1)} = \bigcup_{i=1}^m X_i,   \ \  \ \  X^{(2)} =  \bigcup_{i<j} X_{ij} ,  \  \ \  \  X^{(3)}  =  \bigcup_{i<j<k} X_{ijk}.$$

\begin{remark} \label{normalcrossing}
For our purposes we will be able to assume that the $F_i \in S_n(d_i)$ are general. We shall assume, more precisely, that
each $F_i$ is smooth irreducible and that  $X^{(1)}$ is a normal crossings divisor. Hence, each
${X}_A$ is a smooth complete intersection of codimension $|A|$, and thus the strata $X^{(r)}$ are of codimension $r$,
singular only along $X^{(r+1)}$.
\end{remark}

It is shown in \cite{CSV} and \cite{CS} that for $\omega$ logarithmic as above, with all $\lambda_i \ne 0$, 
$$S(\omega) = X^{(2)} \cup P$$
with $P \subset \P^n - X^{(1)}$ closed,  and $P$ is a finite set if $\omega$ is general. Let's revisit the argument, under the assumptions of Remark \ref{normalcrossing}.
First, since clearly $\hat{F}_i$ vanishes on $X^{(2)}$ for all $i$, we have $X^{(2)} \subset S(\omega)$. 
Since $\omega = \lambda_i \hat{F}_i dF_i$ on $X_i$, we see that $(X^{(1)}-X^{(2)}) \cap S(\omega) = \emptyset$. As for the zeros of $\omega$ in the complement 
of $X^{(1)}$, they are the same as the zeros of $\eta = \omega / F = \sum_{i=1}^m \lambda_i   \  dF_i /F_i $, which is a section
of the locally free sheaf $E = \Omega^1_{\P^n} (\mathrm{log} \ X^{(1)})$ of rank $n$ (see \cite{D}, \cite{GH}, \cite{PS}, \cite{EV}). 
Considering the $F_i$ (hence the divisor $X^{(1)}$) as fixed,
the space of global sections of $E$ has dimension $m-1$, and these sections correspond bijectively with
the residues $(\lambda_1, \dots, \lambda_m)$, satisfying $\sum_i d_i \lambda_i =0$, as it follows from taking cohomology in the exact sequence (\cite{D} or \cite{EV}, p. 170):
$$0 \to \Omega^1_{\P^n}  \to E \to \oplus_{i=1}^m \cO_{X_i} \to 0 .$$
For general $(\lambda_1, \dots, \lambda_m)$ as above, the corresponding section $\eta$ of $E$ has a finite set $P$ of simple zeros. 
Further, the cardinality of $P$ (see \cite{CSV}) is the degree of the top Chern class $c_n(E)$, computable from the exact sequence above.

\medskip \medskip \medskip

Coming back to the study of the resolution of the ideal $\cI_{\omega}$, let us denote
$$\cJ^{(r)} = \cI(X^{(r)}) \subset \cO_{\P^n}$$
the ideal sheaf of regular functions vanishing on $X^{(r)}$, and
$$J^{(r)} = \bigoplus_{k \in \Z} \ H^0(\P^n, \cJ^{(r)}(k))  \subset S_n$$
the corresponding saturated homogeneous ideal.

\medskip
Our arguments to prove stability of logarithmic forms will rely on the following results regarding the ideals $J^{(2)}$.

\begin{proposition} \label{proposition3} Under the hypothesis of Remark \ref{normalcrossing},

a) $J^{(2)}$ is generated by $\{ \hat{F}_i, \ 1 \le i \le m\}$.

b) The relations among the generators of a) are generated by
$$F_j \   \hat{F}_j  - F_i \ \hat{F}_i, \ \  1 \le i < j \le m,$$
and also by the subset
$$R_j =  F_j \ \hat{F}_j - F_1 \  \hat{F}_1, \ \  2 \le j \le m.$$

c) We have a resolution of $\cJ^{(2)}$
$$0 \to \cO(-d)^{m-1} \stackrel{\delta_0} {\longrightarrow} \bigoplus_{1 \le i \le m} \cO(-\hat{d}_i) \stackrel{\delta_1} {\longrightarrow} \cJ^{(2)} \to 0$$
where, denoting $\{e_i\}$ the respective canonical basis,  
$$\delta_0(e_j) =  F_j \ {e}_j -  F_1 \ {e}_1 \ \ \mathrm{for} \  \  2 \le j \le m,$$
$$\delta_1(e_i) = \hat{F}_i \ \ \mathrm{for} \  \ 1 \le i \le m.$$
\end{proposition}

\begin{proof}

a)  We are assuming that the $F_i$ are generic. This implies in particular that each ideal $<F_i, F_j>$ is prime.
Then, $J^{(2)} = \bigcap_{1 \le i < j \le m} <F_i, F_j>$.  Let us denote $J = <\hat{F}_1, \dots, \hat{F}_m>$. 
It is clear that $J \subset J^{(2)}$. We shall prove that $J^{(2)}  \subset  J $ by induction on $m$. The case $m=2$ is trivial.
The inductive hypothesis, applied to $F_1, \dots, F_{m-1}$,  may be written as 
$\bigcap_{1 \le i < j \le m-1} <F_i, F_j> \ \subset \ <\hat{F}_{1m}, \dots, \hat{F}_{m-1 m}>$.
Take an element $G \in  \bigcap_{1 \le i < j \le m} <F_i, F_j> \ = \  \bigcap_{1 \le i < j \le m-1} <F_i, F_j> \  \cap  \
 \bigcap_{1 \le i < m} <F_i, F_m>$.
Using the inductive hypothesis, we may write $G = \sum_{i < m} a_i \hat{F}_{im}$, and we also have $G \in <F_i, F_m>$ for $i < m$.
Since $\hat{F}_{jm} \in <F_i, F_m>$ for $j \ne i$, it follows that $a_i \hat{F}_{im}  \in <F_i, F_m>$ for $i < m$. Since $<F_i, F_m>$ is prime,
we have $a_i = b_i F_i + c_i F_m$. Then, $G = \sum_{i < m} (b_i F_i + c_i F_m) \hat{F}_{im} = \sum_{i < m} (b_i \hat{F}_{m}+ c_i \hat{F}_{i}) \in J$, 
as wanted.

\medskip

\noindent
b) and c) Using the relations $R_j$ of b) we write down the complex in c). The proof will be complete if we show that this complex is exact.
The surjectivity of $\delta_1$ follows from a). Looking at the matrix of $\delta_0$ it is easy to see that
the determinant of the minor obtained by removing row $j$ is precisely $\hat {F}_j$, for $j=1, \dots, m$. Then this complex is the one associated
to the maximal minors of a matrix of size $m \times m-1$. Since in our case, by a), the ideal of minors vanishes in codimension two, the complex is exact
(see \cite{Ar} (5), \cite{E} (20.4)).
\end{proof}

\medskip

\begin{remark} \label{remark vanishing}
Let $X$ be an algebraic variety, $\cJ \subset \cO_X$ a sheaf of ideals,  and $E$ a locally free sheaf on $X$. Let $Y \subset X$ denote the subvariety
corresponding to $\cJ$. Taking global sections on the exact sequence $0 \to E \otimes \cJ \to E \to E \otimes \cO_Y = E|_Y \to 0$ we obtain an identification of
$H^0(X, E \otimes \cJ)$ with the global sections of $E$ vanishing on $Y$, that is, with the kernel of the restriction map $H^0(X, E) \to H^0(Y, E|_Y)$.
\end{remark}

\medskip

\begin{proposition} \label{expression} Let $\alpha \in \Omega^1_n(d)$ be a 1-form of degree $d$ in $\C^{n+1}$. 
Denote $\tilde X^{(2)}  \subset \C^{n+1}$ the cone over $X^{(2)}$.

a) $\alpha$ vanishes on  $\tilde X^{(2)}$ if and only if it may be written as
$$\alpha = \sum_{i=1}^m  \hat F_i  \alpha_i$$
for some $\alpha_i \in \Omega^1_n(d_i)$.

b) $\alpha$ is projective (see Section \ref{notation}) and vanishes on  $X^{(2)}$ if and only if it may be written as
$$\alpha = \sum_{i=1}^m  \lambda'_i  \hat F_i  dF_i  +  \sum_{i=1}^m   \hat F_i  \gamma_i $$
where $\lambda'_i \in \C$,  $\sum_{i=1}^m d_i \lambda'_i = 0$ and $\gamma_i \in H^0(\P^n, \Omega^1_{\P^n}(d_i))$
are projective 1-forms of respective degrees $d_i$.

\end{proposition}

\begin{proof}
\noindent
a) By Remark \ref{remark vanishing}, we need to determine $H^0(\P^n, \Omega^1_{\P^n}(d) \otimes \cJ^{(2)})$.
The stated result then follows from Proposition \ref{proposition3} c), by tensoring with $\Omega^1_{\P^n}(d)$ and taking global sections.

\noindent
b) Suppose $\alpha$ is also projective, that is, $<R, \alpha> = 0$, where $R$ is the radial vector field. From a) we have
$$\sum_{i=1}^m  \hat F_i  <R, \alpha_i> = 0.$$
This is a relation among the $\hat F_i$ with coefficients $<R, \alpha_i>$ homogeneous of degrees $d_i$.
By Proposition \ref{proposition3} c), by tensoring with $\cO_{\P^n}(d)$ and taking global sections, 
this relation is a linear combination of the relations $R_i$ of Proposition \ref{proposition3} b), that is,
$$(<R, \alpha_1>, \dots, <R, \alpha_m>)= \sum_{2 \le i \le m} a_i R_i.$$
\noindent
This means that 
$$<R, \alpha_1> = (\sum_j a_j) F_1, \ \ \   <R, \alpha_i> = -a_i F_i, \ \  i = 2, \dots, m.$$
Hence $a_i$ has degree zero, i. e. $a_i \in \C$, for all $i$. Define $\lambda'_i = a_i / d_i$ for $i = 2, \dots, m$, 
$\lambda'_1 = - (\sum_j a_j) / d_1$
and $\gamma_i = \alpha_i - \lambda'_i dF_i$.
It follows that $<R, \gamma_i> = 0$ and hence $\alpha$ may be written as stated.
\end{proof}

\medskip

\medskip

\section{Surjectivity of the derivative and main Theorem.}    \label{surjectivity}

As in Remark \ref{remark3} we denote the derivative of $\mu$  at the point $\mu(\lam, \F)$
\begin{equation}
d\mu(\lam, \F): V_n(\d) \to T(\omega)  \label{derivative2}
\end{equation}
where $\omega= \mu(\lam, \F)$ and 
\begin{equation}
T(\omega)= T_{F_n(d)}(\omega) = \{\alpha \in  H^0(\P^n, \Omega^1_{\P^n}(d)) / 
\ \omega \wedge d\alpha +  \alpha  \wedge d\omega  = 0\} \label{tangentspace}
\end{equation}
denotes the Zariski tangent space of $F_n(d)$ at the point $\omega$. 

\medskip

Our main objective is to prove the following:

\medskip

\begin{theorem} \label{main} Let $n, d, m$ and $\d \in P(m, d)$ be as in Definition \ref{definition0}. Suppose $n \ge 3$.
Then the derivative $d\mu(\lam, \F): V_n(\d) \to T(\omega)$ is surjective for $(\lam, \F) \in V_n(\d)$ general.
\end{theorem}

\begin{proof}
The proof will be obtained through various steps, including several Propositions of independent interest.
\end{proof}

\medskip

\begin{theorem} \label{main2}
If $n \ge 3$, the set of logaritmic forms $\\L_n(\d) \subset  \mathcal F_n(d)$, as in Definition \ref{definition component}, 
is an irreducible component of $\mathcal F_n(d)$. Furthermore, the scheme $\mathcal F_n(d)$ is \emph{reduced}
generically along  $\\L_n(\d)$.
\end{theorem}

\begin{proof} Follows from Theorem \ref{main} by the same arguments as in \cite{CP}  or \cite{CPV}.
\end{proof}

\medskip
\medskip
Let us now start with several steps towards the proof of Theorem  \ref{main}.

\medskip
\medskip

\begin{remark} \label{remark4} A typical element $\alpha$  in the image of $d\mu(\lam, \F)$ as in \ref{derivative}
$$\alpha =  \sum_{i} \lambda'_i  \  \hat{F}_i  \  dF_i +
\sum_{i \ne j} \lambda_i   \ F'_j \ \hat{F}_{i j}  \  dF_i + \sum_{i} \lambda_i  \  \hat{F}_i  \  dF'_i $$
may be written
$$\alpha =  \sum_{i}  \hat{F}_i \ (\lambda'_i  \  dF_i  + \lambda_i  \  dF'_i)  +
\sum_{i \ne j} \lambda_i   \  F'_j \ \hat{F}_{i j}  \  dF_i $$
or
$$\alpha =  \sum_{i}  \hat{F}_i \  (\lambda'_i   \ dF_i  + \lambda_i   \  dF'_i)  +
\sum_{i < j}  \hat{F}_{i j}  \ ( \lambda_i   \ F'_j \  dF_i + \lambda_j   \ F'_i  \  dF_j) $$
Let us observe that the first sum is zero on  $X^{(2)}$ (hence on  $X^{(3)}$) and the second sum is zero on  $X^{(3)}$.
The idea of our proofs, leading to Theorem \ref{main}, will be based on this observation.
\end{remark}

\medskip

Our strategy to characterize the elements $\alpha \in T(\omega)$ will be this: first we shall determine $\alpha|_{X^{(3)}}$, next we shall determine $\alpha|_{X^{(2)}}$, and finally we show that $\alpha$ may be written as in \ref{derivative} for some $\lam'$ and $\F'$, 
and therefore $\alpha$  belongs to the image of $d\mu(\lam, \F)$.

\medskip

In order to carry out this plan, let us start with some Propositions, some of them of independent interest.

\medskip

\begin{proposition} \label{proposition beta}
For $\omega \in F_n(d)$ and $\alpha \in  H^0(\P^n, \Omega^1_{\P^n}(d))$,
the following conditions are equivalent:

a) $\omega \wedge d\alpha +  \alpha  \wedge d\omega  = 0$, that is, $\alpha \in T(\omega)$.

b) $d\omega \wedge d\alpha = 0$.

\noindent
Further,  for $\omega$ \emph{logarithmic}, $\eta = \omega / F$ and $\beta = \alpha / F$,

c) $\eta \wedge d\beta = 0$.

d) $d(\eta \wedge \beta) = 0$.

\end{proposition}

\begin{proof}
From a) one obtains b) by applying exterior derivative. Conversely, from b) one obtains a) by contracting with the radial vector field.
The equivalence with c) follows from Proposition \ref{proposition1.5} by a straightforward calculation. The equivalence of c) and d) follows from the
fact that $\eta$ is closed.
\end{proof}

\medskip

\begin{proposition} \label{proposition6}
Let $\omega= \mu(\lam, \F)$ be a logarithmic form and $\alpha \in T(\omega)$. Assume that $X^{(1)}$ is normal crossings, with smooth irreducible components $X_i$, as in Remark \ref{normalcrossing}. Then $\alpha|_{X^{(3)}} = 0$, that is, $\alpha(x) = 0$ for all $x \in X^{(3)}$.
\end{proposition}
\begin{proof}
Let us denote, for  $1 \le i<j \le m$,
$$U_{ij} := X_{ij} - X^{(3)} = \{x \in \P^n / F_i(x) = F_j(x) = 0, \ F_k(x) \ne 0 \ \mathrm{for} \  k \notin \{i, j\}\}$$
and, similarly, for $ 1 \le i<j<k \le m$,
$$U_{ijk} := X_{ijk} - X^{(4)}$$
Since the set of zeros of $\alpha$ is closed, it is enough to see that $\alpha$ is zero on $X^{(3)} - X^{(4)}$, which is the disjoint union of the $U_{ijk}$. Notice that $dF_i, dF_j, dF_k$ are linearly independent on $U_{ijk}$ because of the normal-crossings hypothesis. Since clearly $\omega|_{X^{(2)}} = 0$, the relation $\omega \wedge d\alpha +  \alpha  \wedge d\omega  = 0$ reduces to $\alpha(x) \wedge d\omega(x) = 0$ for each $x \in X^{(2)}$. We may assume that $\lambda_i \ne \lambda_j$ for $i \ne j$ without losing generality. 
Then it follows from Proposition \ref{proposition1.5} a) that 
\begin{equation}
\alpha \wedge dF_i \wedge dF_j = 0  \label{a}
\end{equation}
on $U_{ij}$, and hence on its closure $X_{ij}$. This means that 
\begin{equation}
\alpha(x) \in \C.dF_i(x) + \C. dF_j(x) \subset \Omega^1_{\P^n}(x) \label{b}
\end{equation} 
for $x \in X_{ij}$.
Therefore, for $x \in U_{ijk}$ we have 
$$\alpha(x) \in (\C.dF_i(x) + \C.dF_j(x)) \cap (\C.dF_i(x) + \C.dF_k(x)) \cap (\C.dF_j(x) + \C.dF_k(x)).$$
Due to the normal crossings hypothesis this last intersection of two-dimensional subspaces is zero, hence $\alpha(x) = 0$ for $x \in U_{ijk}$, as wanted.
\end{proof}

\medskip

\begin{proposition} \label{proposition7} With the notation and hypothesis of Proposition \ref{proposition6}, for each ordered pair
$(i, j)$ with $1 \le i, j \le m$ and $i \ne j$, there exists $A_{ij} \in S_n(d_j)$ such that
$$\alpha = \hat {F}_{ij} \ (A_{ij} \ dF_i + A_{ji} \  dF_j)   \ \mathrm{on}  \  X_{ij}. $$
\end{proposition}

\begin{proof} This will follow easily combining that $X_{ij}$ is a smooth complete intersection of codimension two in a proyective space, and the fact that $\alpha|_{X^{(3)}} = 0$ that we just proved.

Suppose $J = <A, B>$ is the ideal generated by general homogenous polynomials $A$ and $B$ of respective degrees $a$ and $b$. 
Let $Y \subset \P^n$ be the set of zeroes of $J$. We have an exact sequence (\cite{H}, II.8)
$$0 \to J/J^2 = \cO_Y(-a) \oplus \cO_Y(-b) \stackrel{\delta} {\longrightarrow}   \Omega^1_{\P^n}|_Y \to \Omega^1_Y  \to 0$$
Tensoring with $\cO_Y(d)$ and taking global sections we obtain that an element 
$\alpha|_Y \in H^0(Y, \Omega^1_{\P^n}(d)|_Y)$ which belongs to the image of $H^0(\delta)$, may be written as $A' dA + B' dB$ for
$A' \in H^0(Y, \cO_Y(d-a))$ and $B' \in H^0(Y, \cO_Y(d-b))$.  By \cite{H}, Ex. III (5.5), $A'$ and $B'$ are represented by homogeneous polynomials of respective degrees $d-a$ and $d-b$. 

For each $(i, j)$, $\alpha|_{ X_{ij}}$ belongs to the image of the corresponding $H^0(\delta)$, by \ref{b}.
Hence, we know that $\alpha =  A'_{ij} \ dF_i + A'_{ji} \  dF_j$   on   $X_{ij}$, for homogeneous polynomials
$A'_{ij}$ of degree $d - d_i$. Now, $\alpha|_{X^{(3)}} = 0$ by Proposition \ref{proposition6}, and in  particular $\alpha = 0$ on $X_{ijk}$ for all $k$.
Since $dF_i$ and $dF_j$ are linearly independent at all points of $X_{ijk}$ by the normal crossings hypothesis, it follows that  $A'_{ij}$ and $A'_{ji}$ are divisible by $\hat {F}_{ij}$ and we obtain the claim.
\end{proof}

\medskip

\begin{corollary} \label{corollary prime} With the notation of Proposition \ref{proposition7}, define
$$\alpha' = \sum_{i < j} \hat {F}_{ij} \ (A_{ij} \ dF_i + A_{ji} \  dF_j) \in \Omega^1_n(d)$$
Then $\alpha'|_{\tilde X^{(2)}} = \alpha|_{\tilde X^{(2)}}$. 

\medskip
\noindent
(But notice that $\alpha'$ may not satisfy \ref{tangents}; see the Proof of Corollary \ref{any alpha 3}).
\end{corollary}

\begin{proof} 
Follows from Proposition \ref{proposition7} since $\hat {F}_{ij}$ vanishes on $X_{hk}$ if $\{h, k\} \ne \{i, j\}$.
\end{proof}

\medskip

\begin{corollary} \label{corollary any alpha}  We keep the notation of Proposition \ref{proposition7}.
Then any $\alpha \in T(\omega)$ may be written as
\begin{eqnarray*} 
\alpha &=&  \sum_{i < j} \hat {F}_{ij} \ (A_{ij} \ dF_i + A_{ji} \  dF_j) + \sum_{i}  \hat F_i \  \alpha_i \\
 { } &=& \sum_{i \ne j} \hat {F}_{ij} \ A_{ij} \ dF_i  + \sum_{i}  \hat F_i \  \alpha_i.
 \end{eqnarray*} 
 for some $\alpha_i \in \Omega^1_n(d_i)$.
\end{corollary}

\begin{proof} 
For $\alpha \in T(\omega)$, take $\alpha'$ as in Corollary \ref{corollary prime}. Then $\alpha - \alpha' \in \Omega^1_n(d)$ 
vanishes on $\tilde X^{(2)}$ and hence, by Proposition \ref{expression} a),
may be written as $\sum_{i=1}^m  \hat F_i  \alpha_i$ for some $\alpha_i \in \Omega^1_n(d_i)$. 
\end{proof} 

\medskip

\noindent
We would like to obtain further information on the $A_{ij}$'s and the $\alpha_i$'s. For this, we will use again that $\alpha$ 
satisfies $\omega \wedge d\alpha + \alpha \wedge d\omega = 0$ as in  \ref{tangents}.

\medskip

\begin{proposition} \label{proposition8} Suppose $n \ge 3$. With notation as in Corollary \ref{corollary any alpha}, 
for each $j = 1, \dots, m$ there exists $F'_{j} \in S_n(d_j)$  such that
$$A_{ij} = \lambda_i  \ F'_{j} \ \ \mathrm{\ on \ } X_{ij}$$
for all $(i, j)$  with $1 \le i, j \le m$ and $i \ne j$. 
\end{proposition}

\begin{proof}

The calculation is nicer working with the equivalent condition
 $d\beta \wedge \eta = 0$, where $\beta = \alpha / F$ and $\eta = \omega/F$, see Proposition \ref{proposition beta} c).
 We have:
 $$\beta = \sum_{i \ne j}  \frac{A_{ij}}{ F_j} \ \frac{dF_i}{F_i} + \sum_{i}  \frac{\alpha_i}{F_i}$$
 $$d\beta = \sum_{i \ne j}  d(\frac{A_{ij}}{ F_j})  \wedge \frac{dF_i}{F_i} + \sum_{i}  d(\frac{\alpha_i}{F_i})$$
 \begin{eqnarray*}
 d\beta \wedge \eta \ =  \sum_{i \ne j, k}  \lambda_k \ d(\frac{A_{ij}}{ F_j})  \wedge \frac{dF_i}{F_i}  \wedge \frac{dF_k}{F_k} +
  \sum_{i, k} \lambda_k \ d(\frac{\alpha_i}{F_i})  \wedge \frac{dF_k}{F_k} = \\
 \sum_{i \ne j \ne k}  \lambda_k \ d(\frac{A_{ij}}{ F_j})  \wedge \frac{dF_i}{F_i}  \wedge \frac{dF_k}{F_k} +
  + \sum_{i \ne j}  \lambda_j \ d(\frac{A_{ij}}{ F_j})  \wedge \frac{dF_i}{F_i}  \wedge \frac{dF_j}{F_j} + \\
   \sum_{i \ne k} \lambda_k \ d(\frac{\alpha_i}{F_i})  \wedge \frac{dF_k}{F_k} +
  \sum_{k} \lambda_k \ d(\frac{\alpha_k}{F_k})  \wedge \frac{dF_k}{F_k}= 0
  \end{eqnarray*}
  
  \noindent
 Let's replace 
$$d(\frac{A_{ij}} {F_j}) = \frac{dA_{ij}}{F_j} - \frac{A_{ij}}{F_j} \frac{dF_j}{F_j}, 
\ \ d(\frac{\alpha_i}{F_i}) = \frac{d\alpha_i}{F_i} - \frac{dF_i}{F_i} \wedge \frac{\alpha_i}{F_i}$$ 
and multiply by $F^2$. After some straightforward calculation we obtain:

\begin{eqnarray*} 
F \ \sum_{i \ne j \ne k}  \lambda_k \     \hat {F}_{ijk}  \ dA_{ij}  \wedge dF_i  \wedge dF_k +
 \sum_{i \ne k}  \lambda_k \   \hat {F}_k \ \hat {F}_{ik}  \ dA_{ik}  \wedge dF_i  \wedge dF_k + \\
\sum_{i \ne j \ne k}  \lambda_k  \   \hat {F}_j \  \hat {F}_{ijk} \ A_{ij}  \ dF_i \wedge dF_j  \wedge  dF_k + \\
F \ \sum_{j \ne k}  \lambda_k  \   \hat {F}_{jk}    \  d\alpha_j \wedge dF_k +  
  \sum_{k}  \lambda_k  \   \hat {F}_{k}^2    \  d\alpha_k \wedge dF_k + \\
 \sum_{j \ne k}  \lambda_k  \   \hat {F}_{j} \ \hat {F}_{jk} \ \alpha_j \wedge dF_j \wedge dF_k \ = \ 0
\end{eqnarray*} 

\noindent
Now we choose $r$ such that $1 \le r \le m$  and restrict  to $X_r$, that is, we reduce modulo $F_r$. We get:
\begin{eqnarray} 
\hat {F}_r \ ( \sum_{i \ne r}  \lambda_r \ \hat {F}_{ir}  \ dA_{ir}  \wedge dF_i  \wedge dF_r + 
\sum_{i \ne k \ne r}  \lambda_k  \  \hat {F}_{irk} \ A_{ir}  \ dF_i \wedge dF_r  \wedge  dF_k + \nonumber \\
\lambda_r   \ \hat {F}_{r}   \  d\alpha_r \wedge dF_r + 
 \sum_{k \ne r}  \lambda_k  \  \hat {F}_{rk} \ \alpha_r \wedge dF_r \wedge dF_k) \ = \ 0 \label{BB}
\end{eqnarray} 
Since $\hat {F}_r$ is not zero on the irreducible variety $X_r$, we may cancel this factor out.

\medskip

\noindent
Next, choose $s$ such that $1 \le s \le m$, $s \ne r$,  and further restrict  to $X_r \cap X_s = X_{rs}$ to obtain:
\begin{eqnarray} 
 \lambda_r \ \hat {F}_{sr}  \ dA_{sr}  \wedge dF_s  \wedge dF_r + 
\sum_{k \ne r \ne s}  \lambda_k  \  \hat {F}_{srk} \ A_{sr}  \ dF_s \wedge dF_r  \wedge  dF_k + \nonumber \\
\sum_{i  \ne r \ne s}  \lambda_s  \  \hat {F}_{irs} \ A_{ir}  \ dF_i \wedge dF_r  \wedge  dF_s \ + \ 
\lambda_s  \  \hat {F}_{rs} \ \alpha_r \wedge dF_r \wedge dF_s \ = \ 0  \label{AA}
\end{eqnarray} 

\medskip

\noindent
And, once more, choose $t$ such that $1 \le t \le m$, $t \ne s \ne r$. Restricting  to $X_r \cap X_s \cap X_t= X_{rst}$
we get:
$$ \hat {F}_{rst}  (\lambda_t    \ A_{sr} - \lambda_s   \ A_{tr})  \ dF_r \wedge dF_s  \wedge  dF_t = 0$$
By the genericity of the $F_i$'s,  $X_{rst}$ is irreducible, and we may cancel out the factor $ \hat {F}_{rst} \ne 0$. 
By the normal crossing hypothesis
we may also cancel out $dF_r \wedge dF_s  \wedge  dF_t \ne 0$. 

\medskip

\noindent
Therefore,
\begin{equation}
A_{sr} /  \lambda_s =  A_{tr} / \lambda_t    \  \  \mathrm{\ on \ }   X_{rst} \label{A}
\end{equation}
for all distinct $1 \le r, s, t \le m$.

\medskip

\noindent
Let us fix $r$, $1 \le r \le m$. We consider the natural restriction maps
$$S_n(d_r) = H^0(\P^n, \cO(d_r)) \to H^0(X_{r}, \cO(d_r))  \to  H^0(X_{rs}, \cO(d_r))  \to  H^0(X_{rst}, \cO(d_r)).$$
For $s = 1, \dots, m$, $s \ne r$, the polynomials $A_{sr} / \lambda_s \in S_n(d_r)$ (all of the same degree $d_r$) define, 
by restriction to the hypersurfaces $X_{rs} \subset X_r$, sections $A_{sr} / \lambda_s \in  H^0(X_{rs}, \cO(d_r))$. 
By \ref{A} these sections coincide on the pairwise intersections $X_{rs} \cap X_{rt}  = X_{rst}$.
Hence this collection defines a section of $\cO(d_r)$ on the
(reducible) variety $D_r = \cup_ {s \ne r} X_{rs} \subset X_r$. By  Lemma \ref{lema} below, 
with $X = X_r$ and $D = D_r$, there exists $F'_r \in S_n(d_r)$, such that
$A_{sr} /  \lambda_s = F'_r$ on $X_{rs}$, for each $s \ne r$, as claimed. 

\medskip

\end{proof}

\begin{lemma} \label{lema} Let $n \ge 3$, and let $X \subset \P^n$ be a smooth irreducible hypersurface of degree $e$.
For $m \ge 1$ and  $i =1, \dots, m$ let $D_i \subset X$ be smooth irreducible distinct hypersurfaces. 
We consider the (reducible) hypersurface $D = \cup_{1 \le i \le m} D_i \subset X$. Then the natural restriction map
$$H^0(X, \cO(e)) \to H^0(D, \cO(e))$$
is surjective.
\end{lemma}

\begin{proof} In the exact sequence $0 \to \cO_X(-D) \to \cO_X \to  \cO_D  \to 0$ we tensor  by $\cO_X(e)$
and take cohomology. Since $\cO_X(-D)(e) =  \cO_X(-d)(e) =  \cO_X(e-d) $ for some $d$, and $H^1(X, \cO_X(e-d)) = 0$
(see e. g. \cite{H}, Exercise III, (5.5)), we obtain the claim.
\end{proof}

\medskip

\begin{corollary} \label{any alpha 2} Let $n \ge 3$. Any $\alpha \in T(\omega)$ may be written as 

$$
\alpha =  \sum_{i \ne j} \lambda_i \ \hat {F}_{ij} \ F'_j \ dF_i  + \sum_{i}  \hat F_i \  \alpha_i.
$$
 for some $F'_i \in S_n(d_i)$ and $\alpha_i \in \Omega^1_n(d_i)$. 
\end{corollary}

\begin{proof}
Follows from Corollary \ref{corollary any alpha} and Proposition \ref{proposition8}.
\end{proof}

\medskip

\begin{corollary} \label{any alpha 3} Let $n \ge 3$. Any $\alpha \in T(\omega)$ may be written as 

$$
\alpha =  \bar \alpha + \sum_{i}  \hat F_i \  \gamma_i.
$$
where $\bar \alpha$ belongs to the image of $d\mu(\lam, \F)$,  $\gamma_i \in \Omega^1_n(d_i)$
and $ \sum_{i}  \hat F_i \  \gamma_i \in T(\omega)$.
\end{corollary}

\begin{proof} Using Corollary \ref{any alpha 2}, then adding and substracting  $\sum_{i} \lambda_i  \ \hat{F}_i  \ dF'_i$, we have:
\begin{eqnarray*} 
\alpha &=&  \sum_{i \ne j} \lambda_i \ \hat {F}_{ij} \ F'_j \ dF_i  +  \sum_{i}  \hat F_i \  \alpha_i  \\
 &=& \sum_{i \ne j} \lambda_i \ \hat {F}_{ij} \ F'_j \ dF_i  + \sum_{i} \lambda_i  \ \hat{F}_i  \ dF'_i  +
  \sum_{i}  \hat F_i \  (\alpha_i  - \lambda_i   \ dF'_i)  \\
 &=& d\mu(\lam, \F)(0, \F') + \sum_{i}  \hat F_i \  \gamma_i
\end{eqnarray*} 
taking $\gamma_i = \alpha_i  - \lambda_i   \ dF'_i$. Since $\alpha, \bar \alpha \in T(\omega)$,
we have
$\alpha - \bar \alpha =   \sum_{i}  \hat F_i  \gamma_i \in T(\omega)$, as claimed.
\end{proof}

\medskip

\begin{remark} \label{remark5}  Corollary \ref{any alpha 3} implies that to prove Theorem \ref{main}  we are reduced to showing that any
$\alpha \in T(\omega)$ of the form $\alpha= \sum_i  \  \hat F_i  \gamma_i$,
with $\gamma_i \in  \Omega^1_n(d_i)$, belongs to the image of $d\mu(\lam, \F)$. 
\end{remark}

\medskip

\noindent
To this end, let us first prove the following 
\medskip

\begin{proposition} \label{proposition9} 
Let $\alpha \in T(\omega)$ be of the form 
\begin{equation}
\alpha= \sum_j  (\hat F_j)^e \ \gamma_j \label{alpha simple}
\end{equation}
with $e \in \N, e \ge 1$, and $\gamma_j \in  \Omega^1_n(d - e \hat d_j)$. Then, for $1 \le i, j \le m$, $i \ne j$, 
there exist $\lambda'_{j} \in \C$, $D_{ij} \in S_n(d_j - e \hat d_j)$ and $\epsilon_j \in \Omega^1_n(d_j - e \hat d_j)$,  such that
$$\gamma_j  = \lambda'_j  \ dF_j + \sum_{i \ne j}  \hat F_{ij} \ D_{ij} \ dF_i + \hat F_j \ \epsilon_j$$
for $j = 1, \dots, m$. In case $e \ge 2$, all  $\lambda'_j = 0$.
\end{proposition}

\begin{proof}
Let us use once more that $\alpha$ satisfies  \ref{tangents}
$\omega \wedge d\alpha +  \alpha  \wedge d\omega  = 0. $ 
We may apply to our present $\alpha$ the calculation
in the Proof of Proposition \ref{proposition8}, with $A_{ij} = 0$ and $\alpha_j = (\hat F_j)^{e-1} \ \gamma_j$, for all $i, j$. Then it follows from equation \ref{AA} that
$$\gamma_j \wedge dF_i \wedge dF_j \ = \ 0  \ \ \mathrm{\ \ on \ } X_{ij}, \ \ \mathrm{for \ all \ \ }    i \ne j,  \label{AAA}$$
since $\lambda_j \ne 0$, and $\hat F_{ij} \ne 0$ on $X_{ij}$. Then,
$$\gamma_j = B_{ij} dF_i + C_{ij} dF_j  \ \ \mathrm{\ \ on \ } X_{ij}$$
for some $B_{ij} \in S_n(d - e  \hat  d_j - d_i) $ and $C_{ij} \in S_n((1-e)  \hat  d_j )$. 
Notice that $C_{ij} \in S_n(0) = \C$ if $e=1$, and $C_{ij} = 0$ if $e \ge 2$, since $(1-e)  \hat  d_j  < 0$.

\noindent
Now we fix $j$ and vary $i \ne j$. On $X_{ij} \cap X_{kj} = X_{ijk}$ we  have
$B_{ij} dF_i + C_{ij} dF_j  = B_{kj} dF_k + C_{kj} dF_j$. From the normal crossings hypothesis we obtain,  for all $i \ne k$:

a) $B_{ij} = B_{kj} = 0$ on $X_{ijk}$, and  

b) $C_{ij} = C_{kj}$ 

\noindent
From b),  $C_{ij}$ does not depend on $i$ and we may denote $C_{ij} = \lambda'_j$. As noticed above, $C_{ij} = \lambda'_j = 0$ in case $e \ge 2$.

\noindent
On the other hand, a) implies
that 
$B_{ij} = \hat F_{ij} D_{ij}$ on $X_{ij}$
for some $D_{ij} \in  S_n(d_j - e \hat d_j)$. Therefore,
$$\gamma_j = \lambda'_j dF_j +  \hat F_{ij} D_{ij} dF_i  \ \ \mathrm{\ \ on \ } X_{ij}$$
for all $j$ and all $i \ne j$. Let $\gamma'_j = \gamma_j - (\lambda'_j dF_j + \sum_{i \ne j} \hat F_{ij} D_{ij}  dF_i) \in \Omega^1_n(d - e \hat d_j)$.
Then $\gamma'_j$ is zero on $D_j = \cup_{i \ne j} X_{ij} \subset X_j$, hence there exists $\epsilon_j \in \Omega^1_n(d_j - e \hat d_j)$
such that $\gamma'_j = \hat F_j \ \epsilon_j$ on $X_j$. Denoting  $J_j \cong \cO(-d_j)$ the ideal sheaf of $X_j$, we have
$H^0(\P^n, \Omega^1_{\P^n}(d_j)(J_j)) \cong H^0(\P^n, \Omega^1_{\P^n}) = 0$. Therefore
the equality $\gamma'_j = \hat F_j \ \epsilon_j$ holds in $\P^n$, and this implies our claim.

\end{proof}

\medskip
\medskip

\begin{corollary} \label{divisible} 
If $\alpha \in T(\omega)$ is divisible by $(\hat F_1)^e$, that is, $\alpha= (\hat F_1)^e \ \gamma_1$ for some
$\gamma_1 \in  \Omega^1_n(d - e \hat d_1)$, then  there exist
$\lambda'_{1} \in \C$, $D_{i} \in S_n(d_1 - e \hat d_1)$, for $i > 1$, and $\epsilon_1 \in \Omega^1_n(d_1 - e \hat d_1)$,  such that
$$\alpha = (\hat F_1)^e (\lambda'_1  \ dF_1 + \sum_{i > 1}  \hat F_{i1} \ D_{i} \ dF_i  \ + \ \hat F_1 \ \epsilon_1).$$
In case $e \ge 2$,  $\lambda'_1 = 0$.
\end{corollary}

\begin{proof}
It follows immediately from Proposition \ref{proposition9} applied to the case $\gamma_j = 0$ for $j > 1$.
\end{proof}

\medskip
\medskip

\subsection{End of the proof: balanced case.}    \label{balanced}

\begin{definition} \label{balanced}
Let $\d = (m; d_1, \dots, d_m) \in P(m, d)$. We say that $\d$ is \emph{balanced} if $d_i  <   \sum_{j \ne i}  d_j = \hat d_i$ for all $i = 1, \dots, m$. 
Equivalently, if $2 d_i < d$ for all $i$.
\end{definition}

\noindent
Notice that if $\d$ is not balanced then there exists a \emph{unique} $i$ such that $2 d_i \ge d$. 
Since we normalized $\d$ so that $d_1 \ge d_2 \ge \dots \ge d_m$ (see Definition \ref{definition0}), it follows that
$\d$ is balanced if and only if $2 d_1 < d$.

\medskip
\medskip

\begin{theorem} \label{prop balanced}  Suppose $\d \in P(m, d)$ is balanced. Let $(\lam, \F) \in V_n(\d)$ be general and $\omega = \mu(\lam, \F)$.
Then, for any $\alpha \in T(\omega)$ such that  $\alpha= \sum_i  \hat F_i  \ \gamma_i \label{alpha simple}$,
with $\gamma_i \in  \Omega^1_n(d_i)$, there exists $\lam' = (\lambda'_1, \dots, \lambda'_m) \in \C^m$, with $\sum_{i=1}^m d_i \lambda'_i = 0$, such that
$$\alpha = \sum_{i=1}^m \lambda'_i   \ \hat F_i  \  dF_i.$$ 
 In particular, 
$$\alpha = d\mu(\lam, \F)(\lam', 0)$$
 belongs to the image of $d\mu(\lam, \F)$. 
 \end{theorem}

\begin{proof}
We apply Proposition \ref{proposition9} with $e=1$. Since $\d$ is balanced, $d_j - \hat d_j < 0$ for all $j$ and then 
$D_{ij} =0$ and $\epsilon_j = 0$ for all $i, j$. Hence $\gamma_j = \lambda'_j \ dF_j$ for all $j$, as claimed.
\end{proof}

\noindent
It follows from Remark \ref{remark5} that the proof of Theorem \ref{main} is now complete, if $\d$ is balanced.

\subsection{End of the proof: general case.}    \label{balanced}

\medskip

\medskip
\noindent
When $\d$ is not balanced, Theorem \ref{prop balanced} is not true; we may have an $\alpha \in T(\omega)$ such that  $\alpha|_{X^{(2)}} = 0$ but $\alpha$ is not logarithmic as in 
Theorem \ref{prop balanced}. For example, take $F'_1 = G_1 \ \hat F_1$ where $G_1$ is any homogeneous polynomial of degree
$d_1 - \hat d_1 > 0$, and $F'_j = 0$ for $j > 1$. Then $\alpha = d\mu(\lam, \F)(0, F')$ satisfies this condition, as it easily follows from  
\ref{derivative}. Notice that this $\alpha$ is divisible by $\hat F_1$. 

\medskip

\noindent
In Theorem \ref{prop balanced 2} we will see that
any $\alpha \in T(\omega)$ such that  $\alpha|_{X^{(2)}} = 0$ may be written in a special form that still implies it belongs to
the image of $d\mu(\lam, \F)$.

\begin{definition} \label{unbalanced}
Let $\d \in P(m, d)$. We define 
$$r(\d) = \mathrm{max} \ \{e \in \N/  \ d_1  \ge   e \  \hat d_1\} = [{d_1}/{ \hat d_1}]$$
the integer part of ${d_1}/{ \hat d_1}$.
\end{definition}

\noindent
Notice that $\d$ is balanced when $r(\d)=0$.

\medskip
\medskip

\begin{theorem} \label{prop balanced 2}  Fix $\d \in P(m, d)$. Let $(\lam, \F) \in V_n(\d)$ be general and $\omega = \mu(\lam, \F)$.
Then, any $\alpha \in T(\omega)$ such that   $\alpha= \sum_i  \  \hat F_i  \gamma_i$,
with $\gamma_i \in  \Omega^1_n(d_i)$, may be written as
$$\alpha = d\mu(\lam, \F)(\lam', \F')$$ 
where  $\lam' \in \C^m$ is such that $\sum_{i=1}^m d_i \lambda'_i = 0$,
$F'_j = 0$ for $j > 1$, and  
$$F'_1 = \sum_{e=1}^{r(\d)} \ G_e \ \hat F_1^{\ e}$$
where $G_e$ are homogeneous polynomials of respective degrees $d_1 - e \hat d_1$,  for $e = 1, \dots, r(\d)$.

\end{theorem}

\begin{proof} 
By Proposition \ref{proposition9} with $e=1$,
\begin{equation}
\alpha= \sum_j \lambda'_j \  \hat F_j  \ dF_j   +    \sum_{i \ne j}  \hat F_{ij} \ \hat F_j D_{ij} \ dF_i   +  
 \sum_j  \hat F_j\  \hat F_j  \epsilon_j.   \label{alpha simple 2}
\end{equation}
In the current unbalanced case, $d_1 - \hat d_1 \ge 0$ and $d_i - \hat d_i <  0$ for $i > 1$, 
as in Definition \ref{balanced}. Hence $D_{ij} = 0$ and $\epsilon_j = 0$ for $j > 1$. Also, since 
$\sum_j \lambda'_j \  \hat F_j  \ dF_j  = d\mu(\lam, \F)(\lam', 0)$, it is enough to consider
\begin{equation}
\alpha = \alpha^{(1)} =  \sum_{i > 1}  \hat F_{i1} \ \hat F_1 D_{i1} \ dF_i   +   \hat F_1\  \hat F_1  \epsilon_1 =
 \hat F_1 \  ( \sum_{i > 1}  \hat F_{i1} \ D_{i1} \ dF_i   +   \hat F_1\  \epsilon_1) \label{alpha 1}
 \end{equation}
 which is divisible by $\hat F_1$ (the last term is actually divisible by $\hat F_1^{\ 2}$). 
 
 \noindent
 What we shall do is to express $\alpha^{(1)} $ as the sum of
 an element of the image of $d\mu(\lam, \F)$ (of the claimed shape) plus an $\alpha^{(2)} \in T(\omega)$
 divisible by $\hat F_1^{\ 2}$. Next we repeat the argument and express $\alpha^{(2)}$ as the sum
 of another element of the image of $d\mu(\lam, \F)$  plus an $\alpha^{(3)} \in T(\omega)$
 divisible by $\hat F_1^{\ 3}$. After at most $r(\d)$ iterations this process ends, since 
 $\alpha^{(r(\d) + 1)} = 0$ by degree reason, and hence we obtain the claimed expression for the original $\alpha$.
 
 \noindent
 The essential step is to pass from $\alpha^{(e)}$  to $\alpha^{(e+1)}$, for $1 \le e \le r(\d)$. 
 
 \noindent
 To carry out this step, let us assume 
 that $\alpha$  is divisible by $\hat F_1^{\ e}$, that is,
\begin{equation}
\alpha = \alpha^{(e)} = \hat F_1^{\ e} \  ( \sum_{i > 1}  \hat F_{i1} \ D_{i1} \ dF_i   +   \hat F_1\  \epsilon_1).  \label{alpha e}
\end{equation}
as in Corollary \ref{divisible}.

 \noindent
Now we apply to $\alpha$ the calculation in the Proof of Proposition \ref{proposition8}
with 
$$A_{ij} = \hat F_1^{\ e}  D_{ij}, \ \ \alpha_j =  \hat F_1^{\ e}  \epsilon_j,$$ 
that is:
$$A_{i1} = \hat F_1^{\ e}  D_{i1} \  \mathrm{\ for \ } i>1, \ \ \alpha_1 =  \hat F_1^{\ e}  \epsilon_1,$$
$$A_{ij} = 0 ,  \ \ \alpha_j = 0 \ \  \mathrm{\ for \ } j > 1.$$
From equation \ref{BB} with $r=1$ we get

\begin{eqnarray} 
\hat {F}_1 \ ( \sum_{i \ne 1}  \lambda_1 \ \hat {F}_{i1}  \ d(\hat F_1^{\ e}  D_{i1})  \wedge dF_i  \wedge dF_1 + 
\sum_{i \ne k \ne 1}  \lambda_k  \  \hat {F}_{i1k} \ \hat F_1^{\ e}  D_{i1}  \ dF_i \wedge dF_1  \wedge  dF_k + \nonumber \\
\lambda_1   \ \hat {F}_{1}   \  d( \hat F_1^{\ e}  \epsilon_1) \wedge dF_1 + 
 \sum_{k \ne 1}  \lambda_k  \  \hat {F}_{1k} \  \hat F_1^{\ e}  \epsilon_1 \wedge dF_1 \wedge dF_k) \ = \ 0 \label{BB2}
\end{eqnarray}  
We have $d(\hat F_1^{\ e}  D_{i1}) = e  \hat F_1^{\ e-1}  D_{i1} d\hat F_1 + \hat F_1^{\ e}  dD_{i1}$. 
Also, $d\hat F_1 \wedge dF_i =  (\sum_{j \ne 1} \hat F_{j1} dF_j) \wedge dF_i = \sum_{j \ne 1, j \ne i} \hat F_{j1} dF_j \wedge dF_i $, so that
$ \hat {F}_{i1}  d\hat F_1 \wedge dF_i =  \sum_{j \ne 1, j \ne i}  \hat {F}_{i1}  \hat F_{j1} dF_j \wedge dF_i =
 \hat {F}_{1} \ \sum_{j \ne 1, j \ne i}   \hat F_{ij1} dF_j \wedge dF_i$. Replacing these into \ref{BB2}, we obtain, on $X_1$:

 \begin{eqnarray}
 \hat F_1^{\ e+1} (\sum_{i \ne j \ne 1}  e  \lambda_1 \hat {F}_{ij1}  D_{i1} \ dF_j  \wedge dF_i  \wedge dF_1 + 
 \sum_{i \ne 1}  \lambda_1  \hat {F}_{i1} \  dD_{i1} \wedge dF_i \wedge dF_1 +  \nonumber \\
\sum_{i \ne j \ne 1}   \lambda_j   \hat {F}_{ij1}  D_{i1}  \ dF_i \wedge dF_1  \wedge  dF_j +
e  \lambda_1   \ d\hat {F}_{1} \wedge  \epsilon_1 \wedge dF_1 +  \lambda_1   \hat {F}_{1} \ d\epsilon_1 \wedge dF_1 +  \nonumber \\
 \sum_{i \ne 1}  \lambda_i    \hat {F}_{1i} \   \epsilon_1 \wedge dF_1 \wedge dF_i) \ = \ 0 \label{BB3}
 \end{eqnarray}  
 
 \noindent
 Now we cancel the factor $\hat F_1^{\ e+1}$ on $X_1$ and then restrict to $X_{1st}$ for $1, s, t$ distinct. After straightforward calculation
 we obtain, on $X_{1st}$:
 $$(e \lambda_1 + \lambda_s) D_{t1}  = (e \lambda_1 + \lambda_t) D_{s1}$$
 Then the collection $\{ D_{s1} / (e \lambda_1 + \lambda_s) \in S_n(d_1 - e \hat d_1)\}_{s \ne 1}$ defines a 
 section of $\cO(d_1 - e \hat d_1)$ on $\cup_{s \ne 1} X_{1s} \subset X_1$. Hence, there exists $G_e \in S_n(d_1 - e \hat d_1)$
 such that 
 $$D_{s1} = (e \lambda_1 + \lambda_s)  G_e$$
 on $X_{1s}$ for all $s \ne 1$. Then, with the notation of \ref{alpha e},
 $$\sum_{i > 1}  \hat F_{i1} \ D_{i1} \ dF_i   \ + \   \hat F_1\  \epsilon_1 \  - \  \sum_{i > 1}  \hat F_{i1} \ (e \lambda_1 + \lambda_i)  G_e \ dF_i \ = \ 0$$
 on  $\cup_{s \ne 1} X_{1s} \subset X_1$, and hence is divisible by $\hat F_1$. We obtain
 \begin{equation}
 \alpha =  \hat F_1^{\ e} \ \sum_{i > 1}  \hat F_{i1} \ (e \lambda_1 + \lambda_i)  G_e \ dF_i  \ +  \  \hat F_1^{\ e+1} \  \bar \epsilon_1 \label{alpha e+1}
 \end{equation}
 for some $\bar \epsilon_1 \in  \Omega^1_n(d_1 - e \hat d_1)$.
 
 \noindent
Denote $\F' = (\hat F_1^{\ e} \ G_e, 0, \dots, 0)$.  Combining  \ref{alpha e+1} with 
$$d\mu(\lam, \F)(0, \F') = \sum_{i > 1}  \lambda_i \ F_1^{\ e} \ G_e \  \hat F_{i1}  \ dF_i  + \lambda_1 \hat F_1 d(\hat F_1^{\ e} G_e)$$ 
(see \ref{derivative}),
one immediately obtains
$$\alpha = d\mu(\lam, \F)(0, \F') + \alpha^{(e+1)}$$
with $\alpha^{(e+1)} = \hat F_1^{\ e+1}  \ (\bar \epsilon_1 - \lambda_1 dG_e)$. 
Now, $\alpha^{(e+1)} \in T(\omega)$ because $\alpha$ and $d\mu(\lam, \F)(0, \F')$ belong to $T(\omega)$. 
Since $\alpha^{(e+1)}$ is divisible by $\hat F_1^{\ e+1}$, by Corollary \ref{divisible}, it may be written as in \ref{alpha e} with exponent $e+1$.
Hence we may apply again the previous procedure to  $\alpha^{(e+1)}$.
This proves the essential iterative step and implies our statement.
\end{proof}

\medskip
\medskip

\noindent
It follows from Remark \ref{remark5} that the proof of Theorem \ref{main} is now complete, for any $\d$.

\vspace {8 mm}

\begin{flushleft}
\begin{small}
  Universidad de Buenos Aires and CONICET. \newline
  Departamento de Matem\'{a}tica, FCEN. \newline
  Ciudad Universitaria.\newline
  (1428) Buenos Aires.\newline
  ARGENTINA.

\medskip

\medskip

\noindent
 Fernando Cukierman, fcukier{@}dm.uba.ar
 
 \noindent
 Javier Gargiulo Acea, jngargiulo{@}gmail.com 

\noindent
C\'esar Massri, cmassri{@}dm.uba.ar

\end{small}
\end{flushleft}

\end{document}